\documentclass{amsart}
\usepackage{times}
\usepackage{graphicx}
\usepackage{amssymb,amsmath}
\usepackage{subfigure}
\usepackage[letterpaper=true,colorlinks=true,linkcolor=red,filecolor=green,citecolor=red,pdfpagemode=None]{hyperref}
\newtheorem{theorem}{Theorem}[section]
\newtheorem{lemma}[theorem]{Lemma}

\theoremstyle{definition}

\theoremstyle{remark}
\newtheorem{remark}[theorem]{Remark}

\numberwithin{equation}{section}

\begin{document}

 \title[Convergence results with natural norms]{Convergence results with natural norms: Stabilized Lagrange multiplier method for elliptic interface problems}

%    Only \author and \address are required; other information is
%    optional.  Remove any unused author tags.

%    author one information
% \author[short version for running head]{name for top of paper}
\author{Sanjib Kumar Acharya}
\address{The LNM Institute of Information Technology,
              Jaipur 302031, Rajasthan, India}
\curraddr{}
\email{acharya.k.sanjib@gmail.com}
\thanks{}

%    author two information
\author{Ajit Patel}
\address{The LNM Institute of Information Technology,
              Jaipur 302031, Rajasthan, India}
\curraddr{}
\email{ajit.iitb@gmail.com}
\thanks{}

%    \subjclass is required.
\subjclass[2010]{
65N30,  35J15, 35J25 }

\date{}

\dedicatory{}

%    Abstract is required.
\begin{abstract}
A stabilized Lagrange multiplier method for second order elliptic interface problems is presented
in the framework of mortar method. The requirement of  LBB (Ladyzhenskaya-Babu\v{s}ka-Brezzi)
condition  for mortar method is alleviated by introducing penalty terms in the formulation.
Optimal convergence results are established in natural norm which is independent of mesh.
Numerical experiments are conducted in support of  the theoretical derivations.
\end{abstract}

\maketitle
\section{Introduction}
The best part of  considering Lagrange multiplier formulation is: it converts a constraint problem
in to an unconstrained problem which is  comparably an easy way for implementation (see \cite{babuska}).
Also, we can evaluate both primal and flux variable simultaneously.
On the other hand, one major difficulty in considering the  Lagrange multiplier method is:
the finite dimensional problems have to obey the inf-sup condition (LBB condition)
which rejects many natural choices for approximation.
Fortunately this requirement  has been alleviated  by Barbosa and Hughes (see \cite{barbosa}).
They proposed a stabilized multiplier  method which is stable and optimally convergent with
respect to a mesh-dependent norm. In \cite{barbosa1} the convergence results of these methods are
established with natural norms. Nitsche had  introduced   a penalty term on the boundary to derive
optimal estimates for approximating elliptic problems with nonhomogeneous Dirichlet
boundary condition without enforcing boundary condition on the finite element spaces in \cite{nitsche}.

These ideas are extended to multi-domain problems with non-matching grids by Hansbo
\textrm{et al.} in \cite{Hansbo} and Becker \textrm{et al.} in \cite{becker}.
Wherein, the optimal convergence results are established
in mesh-dependent norm. In \cite{Hansbo},  a
stabilization method has been proposed, which uses
global polynomials as multipliers to avoid the cumbersome integration of products of
unrelated mesh functions and derived the stability under the
condition that the approximation space for the interface multiplier
contains the constant.  A Lagrange multiplier method with penalty  for multi-domain
problems with non-matching  grid  is discussed by Patel
(see \cite{patel2}) which is well-posed and stable but due to
inconsistency there is loss of accuracy.  Stenberg has pointed
out a close connection between Nitsche's method and stabilized schemes and
proposed it as mortaring Nitsche method (see \cite{stenberg}). For a detail study, we refer to
\cite{becker,bernardy,belagacem-h,Hansbo,stenberg}.

In this paper, we  extend the ideas of
\cite{barbosa1} to multi-domain problems with non-matching grid
and establish the optimal error estimates in natural norm. Here, the multipliers
are simply the nodal basis functions restricted to the interface. Error estimates
are obtained with an assumption that: the multiplier space satisfies
the strong regularity property in the sense of Babu\v{s}ka (see \cite{babuska}).
We give some computational results
in support of the theoretical results.

A  brief outline is as follows. We recall some functional spaces and approximation results
in Section 2. In Section 3 we define the stabilized Lagrange multiplier methods
for an elliptic interface problem and derive the error estimates.
We give a matrix formulation of the method in Section 4. In Section 5, some numerical
experiments are given. Finally, we concluded in Section 6.

\section{Preliminaries}
Let $\Omega$ be an open bounded polygonal domain in  $\mathbb{R}^{2}$  with boundary $\partial\Omega$.
We define $\alpha=(\alpha_{1},\alpha_{2})$ as a $2$-tuple of non-negative integers $\alpha_{i},i=1,2$ and with $|\alpha|=\alpha_{1}+\alpha_{2}$ set
$$D^{\alpha}=\frac{\partial^{|\alpha|}}{\partial x_{1}^{\alpha_{1}}\partial x_{2}^{\alpha_{2}}}\cdot$$
The Sobolev space of order $m$ (see \cite{grisvard})  over $\Omega$ is defined  as
\begin{equation*}
H^{m}(\Omega)=\left\{v\in L^{2}(\Omega):D^{\alpha}v\in L^{2}(\Omega), |\alpha|\leq m\right\}
\end{equation*}
equipped with the norm and semi-norm
\begin{equation*}
{||v||}_{H^{m}(\Omega)}=\left(\sum_{|\alpha|\leq m}\int_{\Omega}|D^{\alpha}v|^{2}dx\right)^{1/2},\hspace{0.2cm}
{|v|}_{H^{m}(\Omega)}=\left(\sum_{|\alpha|= m}\int_{\Omega}|D^{\alpha}v|^{2}dx\right)^{1/2},
\end{equation*}
respectively.
Let $r=m+\sigma$ be a positive real number, where $m$  and  $\sigma$ are the integral and fractional part of $r$ respectively. The fractional Sobolev space $H^{r}(\Omega)$ is defined as
\begin{equation*}
H^{r}(\Omega)=\bigg\{v\in H^{m}(\Omega):\int_{\Omega}\int_{\Omega}\frac{\big(D^{\alpha}v(x)
-D^{\alpha}v(y)\big)^{2}}{|x-y|^{2+2\sigma}}dxdy<\infty,~|\alpha|=m\bigg\}
\end{equation*}
with the norm
\begin{equation}\nonumber
{||v||}_{H^{r}(\Omega)}=\bigg({||v||}^{2}_{H^{m}(\Omega)}
+\sum_{|\alpha|=m}\int_{\Omega}\int_{\Omega}\frac{\big(D^{\alpha}v(x)
-D^{\alpha}v(y)\big)^{2}}{|x-y|^{2+2\sigma}}dxdy\bigg)^{1/2}.
\end{equation}
We shall denote by $H^{r-1/2}(\partial\Omega)$ the space of traces ${v|}_{\partial\Omega}$ over $\partial\Omega$ of the functions $v\in H^{r}(\Omega)$ equipped with the norm $${||g||}_{H^{r-1/2}(\partial\Omega)}=\inf_{v\in H^{r}(\Omega),v{|}_{\partial\Omega}=g}{||v||}_{H^{r}(\Omega)}$$ and $$H^{1}_{0}(\Omega)=\{v\in H^{1}(\Omega):v|_{\partial\Omega}=0\}.$$
Let $H^{-1/2}(\partial\Omega)$ be the dual space of $H^{1/2}(\partial\Omega)$ equipped with the norm $${||\mu||}_{H^{-1/2}(\partial\Omega)}=\sup_{g\in H^{1/2}(\partial\Omega),~g\neq 0}\frac{|\langle\mu,g\rangle_{-1/2,\partial\Omega}|}{||g||_{H^{1/2}(\partial\Omega)}},$$ where $\langle\cdot,\cdot\rangle_{-1/2,\partial\Omega}$ is the duality pairing between $H^{-1/2}(\partial\Omega)$ and $H^{1/2}(\partial\Omega)$.
With $\Gamma^{\ast}\subset \partial\Omega$, let $\tilde{v}$ be an
extension of $v\in H^{1/2}(\Gamma^{\ast})$ by zero to all of
$\partial\Omega$. Then we set $H^{1/2}_{00}(\Gamma^{\ast})$, a subspace
of $H^{1/2}(\Gamma^{\ast})$ as
\begin{eqnarray*}
H^{1/2}_{00}(\Gamma^{\ast})=\{v\in H^{1/2}(\Gamma^{\ast})\colon \tilde{v}\in
H^{1/2}(\partial\Omega)\}.
\end{eqnarray*}
The norm in $H^{1/2}_{00}(\Gamma^{\ast})$ is defined by:
$$\|g\|_{H^{1/2}_{00}(\Gamma^{\ast})}=\displaystyle
\inf_{v\in H^1_{0,\partial\Omega\backslash \Gamma^{\ast}}(\Omega),v_{|_{\Gamma^{\ast}}}=g}{\|v\|}_{H^1(\Omega)}.$$

Let $H^{-1/2}_{00}(\Gamma^{\ast})$ be the dual space of
$H^{1/2}_{00}(\Gamma^{\ast})$. Also $\langle\cdot,\cdot\rangle_{00,\Gamma^{\ast}}$ denote the
duality pairing between $H^{-1/2}_{00}(\Gamma^{\ast})$ and
$H^{1/2}_{00}(\Gamma^{\ast})$ and let the norm on $H^{-1/2}_{00}(\Gamma^{\ast})$
be defined by
\begin{eqnarray*}
\|\varphi\|_{H^{-1/2}_{00}(\Gamma^{\ast})}=\displaystyle\sup_{\mu\in
H^{1/2}_{00}(\Gamma^{\ast}),
\mu\not=0}\frac{|\langle\varphi,\mu\rangle_{00,\Gamma^{\ast}}|}{{\|\mu\|}_{H^{1/2}_{00}(\Gamma)}}.
\end{eqnarray*}
Note that, $H^{-1/2}(\Gamma^{\ast})$ is continuously embedded into $H^{-1/2}_{00}(\Gamma^{\ast})$
(see \cite{belagacem-h}).

\begin{figure}[h]
\includegraphics*[width=6.5cm,height=5.5cm]{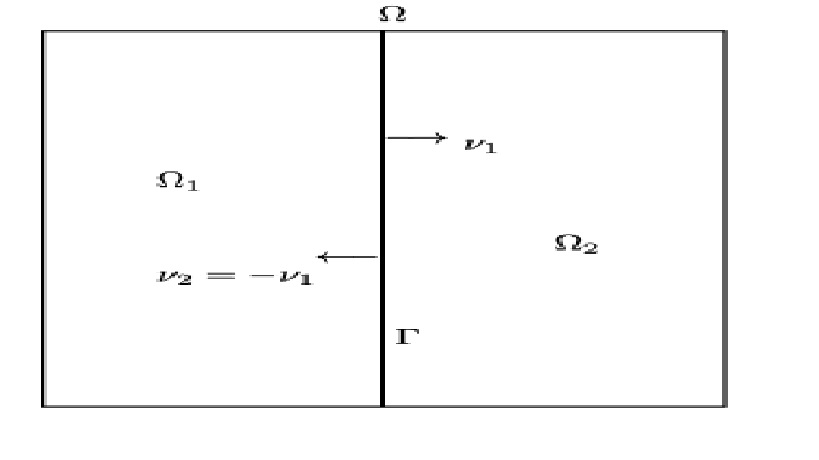}
\centering
\caption{$\bar\Omega=\bar\Omega_1\cup\bar\Omega_2$, $\nu_{1}$ and
$\nu_2$ represent the outward normal components.}
\label{ddmfig2}
\end{figure}

Let $\bar{\Omega}=\bar{\Omega}_{1}\cup\bar{\Omega}_{2}$ and  $\Gamma$ be the common interface $\overline{\partial\Omega}_{1}\cap\overline{\partial\Omega}_{2}$ (see Figure \ref{ddmfig2}). We denote $\nu_{1}$ the unit outward normal oriented from $\Omega_{1}$ towards $\Omega_{2}$ and $\nu_{2}=-\nu_{1}=\nu$. For any function $v$ let $v_{i}=v|_{\Omega_{i}}$. Let $$H^{1}_{D}(\Omega_{i})=\{v_{i}\in H^{1}(\Omega_{i}):v_{i}|_{\partial\Omega\cap\partial\Omega_{i}}=0\}.$$
Now we define $$X=\{v\in L^{2}(\Omega):v_{i}\in H^{1}_{D}(\Omega_{i})~i=1,2\}$$
equipped with the norm $${||v||}_{X}=\left(\sum^{2}_{i=1}{||v||}^{2}_{H^{1}(\Omega_{i})}\right)^{1/2}$$  and the multiplier space $M=H_{00}^{-1/2}(\Gamma)$.

Let $\mathcal{T}_{h_{i}}$ be a family of triangulation (in the sense of Ciarlet, see \cite{ciar}) of $\Omega_{i},i=1,2$ with triangles or parallelograms.
For $K\in\mathcal{T}_{h_{i}},i=1,2$ let $h_{K}=\text{diam}~K$, $\sigma_{K}=\sup\{\text{diam}~B:B ~\text{a ball in}~ K\}$, and $h_{i}=\max{h_{K}}$. Suppose   there exist  positive constants  $\kappa$ and $\varrho$  independent of $h_{i}$ such that for all $K\in\mathcal{T}_{h_{i}},i=1,2$ verifies the quasiuniform condition  $\frac{h_{i}}{h_{K}}\leq\kappa$ and the shape regularity $\frac{h_{K}}{\sigma_{K}}\leq\varrho$.

We define finite dimensional subspaces on each subdomain $\Omega_{i}$  as $$X_{h_{i}}=\{v_{h_{i}}\in C^{0}(\bar{\Omega}_{i}):v_{h_{i}}|_{K}\in\mathcal{P}_{1}(K)\hspace{0.2cm}for\hspace{0.2cm}K
\in\mathcal{T}_{h_{i}},v_{i}=0\hspace{0.2cm}on\hspace{0.2cm}\partial\Omega_{i}\cap\partial\Omega\}.$$ We set  $h=\max\{h_{1},h_{2}\}$ and define the global space $X_{h}\subset X$ defined as
$$X_{h}=\{v_{h}\in L^{2}(\Omega):v_{h_{i}}\in X_{h_{i}}~\text{for}~i=1,2\}.$$

Let $W_{h_{i}}$ be the restriction of $X_{h_{i}}$ to $\Gamma_{i}=\partial\Omega_{i}\cap\Gamma$. Assume two different 1D triangulations on $\Gamma$, $\mathcal{T}_{h_{1}}(\Gamma)$ and $\mathcal{T}_{h_{2}}(\Gamma)$ and correspondingly two different trace spaces   $W_{h_{1}}$ and $W_{h_{2}}$. For our convenience we may choose the multiplier space $W_{h}$ to be  $W_{h_{2}}$.
Now we recall the following approximation results (see \cite{ciar,patel2}).

\begin{lemma}
\label{lem2}
For all $v_{h_{i}}\in X_{h_{i}}$ there exists a constant $C_{I}>0$ independent of $h$  such that
\begin{equation}{||h_{i}^{1/2}\nabla v_{h_{i}}\cdot \nu_{i}||}_{L^{2}({\Gamma}_{i})}
\leq C_{I}{||\nabla v_{h_{i}}||}_{L^{2}(\Omega_{i})}.\end{equation}
\end{lemma}
We assume  $C$ denotes a generic constant throughout the discussion.
\begin{lemma}
\label{inter}
 Let $v_{i}\in H^{l}(\Omega_{i})$ for $l>1$ and $i=1,2$. Then there exist constants $C>0$ independent of  $h$  and a sequence $I_{h_{i}}v_{i}\in X_{h_{i}}$ such that for any $0\leq l_{1}\leq l$
\begin{align}
&{||v_{i}-I_{h_{i}}v_{i}||}_{H^{l_{1}}(\Omega_{i})}\leq C~h_{i}^{l-l_{1}}{||v_{i}||}_{H^{l}(\Omega_{i})},\label{s30}\\
&{||v_{i}-I_{h_{i}}v_{i}||}_{L^{2}(\Gamma_{i})}\leq C~h_{i}^{l-1/2}{||v_{i}||}_{H^{l}(\Omega_{i})}.\label{s31}
\end{align}
\end{lemma}
For $v\in X$ we set the interpolation  $I_{h}v$  as: $I_{h}v$ equals $I_{h_{i}}v_{i}$ on each $\Omega_{i}$ for $i=1,2$.  We define the $L^{2}$ projection $\Pi$ from $M$ onto $W_{h}$  as below:

for all $\phi\in M$,
\begin{equation}
\int_{\Gamma}(\phi-\Pi \phi)\chi=0\hspace{0.2cm}\forall\chi\in W_{h}.
\end{equation}
\begin{lemma}\cite{bernardy}
\label{proj}
For any $\sigma\geq 0$, the following estimate holds: for all $\phi\in H^{1/2+\sigma}(\Gamma)$ there exists a constant $C>0$ independent of $h$ such that
\begin{equation}\label{s15}
h^{1/2}{||\phi-\Pi \phi||}_{L^{2}(\Gamma)}+{||\phi-\Pi \phi||}_{H^{-1/2}(\Gamma)}
\leq C~h^{\eta+1}{||\phi||}_{H^{1/2+\sigma}(\Gamma)}
\end{equation}
where $\eta=\min{(\sigma,1)}$.
\end{lemma}
\section{Problem formulation and error estimates}
Consider a second order elliptic interface model problem: for $i=1,2$
\begin{align}
-\nabla\cdot(\beta_{i}(x)\nabla u_{i})+a_{i}(x)&u_{i} =f\hspace{0.3cm}\text{in}~\Omega_{i},\label{s1}\\
&u_{i}=0\hspace{0.3cm}\text{on}~\partial\Omega_{i}\cap\partial\Omega,\label{s3}\\
[\![u]\!]=0,~[\![\beta\nabla u\cdot\nu&]\!]=0\hspace{0.2cm}\text{along}~\Gamma,\label{s4}
\end{align}
where $\beta$  is  discontinuous along $\Gamma$ but piecewise smooth in each subdomain,  $f$ is an appropriate smooth function, $m_{l}\leq\beta_{i}(x),a_{i}(x)\leq m_{u}$  for some positive constants $m_{l}$ and $m_{u}$ and for all $x\in\overline{\Omega}$, $[\![v]\!]=v_{1}-v_{2}$. Also $\{\!\!\{v\}\!\!\}=\frac{1}{2}(v_{1}+v_{2})$.

The mixed formulation of  \eqref{s1}-\eqref{s4} is to seek a pair
 $(u,\lambda)\in X\times M$ such that
\begin{equation}
a(u,v)+b(v,\lambda)+b(u,\mu)=\mathcal{F}(v)\hspace{0.2cm}\forall ~(v,\mu)\in X\times M,\label{s5}
\end{equation}
 $$\text{where}\hspace{0.2cm}a(v,w)=\sum^{2}_{i=1}\int_{\Omega_{i}}(\beta_{i}\nabla v_{i}\cdot\nabla w_{i}+a_{i}v_{i}w_{i})~dx,~\mathcal{F}(v)=\int_{\Omega}fv~dx,$$
$$b(v,\mu)=\langle\mu,[\![v]\!]\rangle_{00,\Gamma}$$

and $\lambda=\beta_{1}\nabla u_{1}\cdot\nu_{1}=-\beta_{2}\nabla u_{2}\cdot\nu_{2}$ is the Lagrange multiplier.
We note that the bilinear forms $a(\cdot,\cdot)$ and $b(\cdot,\cdot)$ are continuous in $X\times X$ and $X\times M$ respectively.

The stabilized Nitsche's mortaring approximation is: find $(u_{h},\lambda_{h})\in X_{h}\times W_{h}$  such that
\begin{equation}\label{s16}
\mathcal{A}(u_{h},\lambda_{h};v_{h},\mu_{h})=\mathcal{F}(v_{h})\hspace{0.2cm}\text{for all}~ (v_{h},\mu_{h})\in X_{h}\times W_{h},
\end{equation}
where
\begin{align}
\mathcal{A}&(v,\mu;w,\Lambda)=a(v,w)+b(w,\mu)+S\int_{\Gamma}\gamma\mu\big\{\!\!\big\{\beta\nabla w\cdot \nu\big\}\!\!\big\}d\tau+b(v,\Lambda)\nonumber\\
&-S\int_{\Gamma}\gamma\big\{\!\!\big\{\beta\nabla v\cdot \nu\big\}\!\!\big\}\big\{\!\!\big\{\beta\nabla w\cdot \nu\big\}\!\!\big\}d\tau+\int_{\Gamma}\gamma\big\{\!\!\big\{\beta\nabla v\cdot \nu\big\}\!\!\big\}\Lambda~d\tau-\int_{\Gamma}\gamma\mu\Lambda ~d\tau.\label{s9}
\end{align}
Here, $S\in[0,1]$ and $\gamma$ is penalty parameter to be chosen later. When $S\in[0,1)$ the above formulation is unsymmetric and for $S=1$, the formulation is symmetric.
\begin{remark}
In the paper \cite{Hansbo}  a stabilized Lagrange multiplier formulation to circumvent the inf-sup condition has been introduced where they used global polynomials over the interfaces as multipliers  to avoid cumbersome integrations of functions from two different non-matching sides.  We also propose a similar formulation but here we take the trace space as multiplier space, which involves integration of non-matched functions. But our aim here is to derive the error estimates in natural norms. These estimates can be extended for the case: using global polynomials as multipliers as in \cite{Hansbo}.
\end{remark}

For any $S\in[0,1]$ it is easy to check that the problem \eqref{s16} is consistent with the original problem \eqref{s5} and hence the following lemma follows.
\begin{lemma}
The problem \eqref{s16} is consistent with the original problem \eqref{s5}. Moreover, if $(u,\lambda)$ is the solution of \eqref{s5} and $(u_{h},\lambda_{h})$ is the solution of \eqref{s16}, then
\begin{equation}\label{s27}
\mathcal{A}(u-u_{h},\lambda-\lambda_{h};v_{h},\mu_{h})=0\hspace{0.2cm}\text{for all}~ (v_{h},\mu_{h})\in X_{h}\times W_{h}.
\end{equation}
\end{lemma}
\begin{lemma}
\label{coer}
There exists  $\alpha>0$ independent of $h$  such that for all $(v_{h},\mu_{h})\in X_{h}\times W_{h}$:
$\mathcal{A}(v_{h},\mu_{h};v_{h},-\mu_{h})\geq\alpha\bigg(\sum^{2}_{i=1}{||v_{h}||}^{2}_{H^{1}(\Omega_{i})}
+{||\gamma^{1/2}\mu_{h}||}^{2}_{L^{2}(\Gamma)}\bigg)$
for $\gamma=\gamma_{0}h$ with $0<\gamma_{0}<\frac{m_{l}}{C^{2}_{I}m_{u}^{2}}$, $C_{I}$ is a positive constant.
\end{lemma}
\begin{proof}
Taking   $v,w=v_{h}$ and $\mu,\Lambda=\mu_{h}$ in \eqref{s9}, we arrive at
\begin{align}
\mathcal{A}(v_{h},\mu_{h};v_{h},-\mu_{h})
&=\sum^{2}_{i=1}\big(\beta_{i}{||\nabla v_{h_{i}}||}^{2}_{L^{2}(\Omega_{i})}+a_{i}{||v_{h_{i}}||}^{2}_{L^{2}(\Omega_{i})}\big)\nonumber\\
&+(S-1)\int_{\Gamma}\gamma\big\{\!\!\big\{\beta\nabla v_{h}\cdot \nu\big\}\!\!\big\}\mu_{h}~d\tau
-S{||\gamma^{1/2}\big\{\!\!\big\{\beta\nabla v_{h}\cdot \nu\big\}\!\!\big\}||}^{2}_{L^{2}(\Gamma)}\nonumber\\
&+{||\gamma^{1/2}\mu_{h}||}^{2}_{L^{2}(\Gamma)}.\label{s24}
\end{align}
Using Cauchy-Schwarz inequality, Young's inequality, Lemma \ref{lem2}
and using the bounds for $\beta_{i}$, we find the third term in the right hand
side of  \eqref{s24} as

\begin{align}
\bigg|\int_{\Gamma}\gamma\big\{\!\!\big\{\beta\nabla v_{h}\cdot\nu\big\}\!\!\big\}\mu_{h}d\tau\bigg|
\leq\frac{C^{2}_{I}\gamma_{0}m_{u}^{2}}{2}\sum^{2}_{i=1}{|| v_{h_{i}}||}^{2}_{H^{1}(\Omega_{i})}
+\frac{1}{2}{||\gamma^{1/2}\mu_{h}||}^{2}_{L^{2}(\Gamma)}.\label{s25}
\end{align}
Substituting \eqref{s25} in \eqref{s24} and using the bounds for $\beta_{i}$ and $a_{i}$, we find
\begin{align}
\mathcal{A}(v_{h},\mu_{h};v_{h},-\mu_{h})\geq &m_{l}\sum^{2}_{i=1}{||v_{h_{i}}||}^{2}_{H^{1}(\Omega_{i})}
+{||\gamma^{1/2}\mu_{h}||}^{2}_{L^{2}(\Gamma)}-\frac{S C^{2}_{I}\gamma_{0}m_{u}^{2}}{2}
\sum^{2}_{i=1}{||v_{h_{i}}||}^{2}_{H^{1}(\Omega_{i})}\nonumber\\
&(S-1)\bigg(\frac{C^{2}_{I}\gamma_{0}m_{u}^{2}}{2}\sum^{2}_{i=1}{||v_{h_{i}}||}^{2}_{H^{1}(\Omega_{i})}
+\frac{1}{2}{||\gamma^{1/2}\mu_{h}||}^{2}_{L^{2}(\Gamma)}\bigg).\nonumber
\end{align}
Hence the result follows. \qquad
\end{proof}

Note that the uniqueness of the solution of \eqref{s16} is evident from the coercivity property (Lemma \ref{coer})  of $\mathcal{A}(\cdot,\cdot;\cdot,\cdot)$  which establish the existence of the solution.
\vspace{0.1cm}

\begin{theorem}
\label{main}
Let $(u,\lambda)$ and $(u_{h},\lambda_{h})$ be the solutions of \eqref{s5} and
\eqref{s16} respectively with $u_{i}\in H^{2}(\Omega_{i})$. Further, assume $W_{h}$
satisfies the strong regular property in the sense of Babu\v{s}ka (see \cite{babuska}): there exists a constant $C>0$ such that
\begin{equation}\label{streg}
h^{1/2}{||\mu_{h}||}_{L^{2}(\Gamma)}\leq C{||\mu_{h}||}_{H^{-1/2}(\Gamma)}\hspace{0.2cm}\text{for all}\hspace{0.2cm}\mu_{h}\in W_{h}.
\end{equation}

\hspace{-0.5cm}Then for $\gamma=\gamma_{0}h$, there exists a positive constant $C$ independent of  $h$ and $u$ such that
the errors  $e_{u}=u-u_{h}$ and $e_{\lambda}=\lambda-\lambda_{h}$ satisfies:

\begin{align}
{||e_{u}||}^{2}_{X}+{||e_{\lambda}||}^{2}_{M}&\leq C\bigg({||e_{u}||}^{2}_{X}+{||e_{\lambda}||}^{2}_{H^{-1/2}(\Gamma)}\bigg)\nonumber\\
&=C{|||(e_{u},e_{\lambda})|||}^{2}
\leq C~h^{2}\sum^{2}_{i=1}{||u||}^{2}_{H^{2}(\Omega_{i})}.\label{s40}
\end{align}
\end{theorem}
\begin{remark}
The $W_{h}$ space satisfying the strong regularity condition can be constructed using the technique of Hill functions, as in \cite{hill1,hill2} (see \cite{babuska} page 186 for a discussion). 
\end{remark}

In order to prove the above theorem we require  following results:
\begin{lemma}
\label{infsup}
There exist positive constants $C_{1}$ and $C_{2}$ such that for all $\mu_{h}\in W_{h}$,
\begin{equation}\label{s10}
\sup_{0\neq v_{h}\in X_{h}}-\frac{b(v_{h},\mu_{h})}{{||v_{h}||}_{X}}\geq C_{1}{||\mu_{h}||}_{H^{-1/2}(\Gamma)}-C_{2}\gamma_{0}h^{1/2}{||\mu_{h}||}_{L^{2}(\Gamma)}.
\end{equation}
\end{lemma}
\begin{proof}
For $i=1,2$, let $\bar{u}_{i}$ be the solution of the following mixed boundary value problem:
\begin{equation}\label{N}
-\Delta \bar{u}_{i}+\bar{u}_{i}=0,\hspace{0.2cm}\bar{u}_{i}=0\hspace{0.2cm}\text{on}~\partial
\Omega\cap\partial\Omega_{i},\hspace{0.2cm}
\nabla\bar{u}_{i}\cdot\nu_{i}=(-1)^{i+1}\mu_{h}\hspace{0.2cm}\text{on}~\Gamma.
\end{equation}
Then there exist constants $C$ such that:
\begin{equation}\label{s43}
{||\bar{u}_{i}||}_{H^{3/2}(\Omega_{i})}\leq C{||\mu_{h}||}_{L^{2}(\Gamma)}.
\end{equation}
Also from \cite{babuska}, we have
\begin{equation}\label{s44}
{||\bar{u}_{i}||}_{H^{1}(\Omega_{i})}\geq C{||\mu_{h}||}_{H^{-1/2}(\Gamma)}.
\end{equation}

Let $\bar{u}_{h_{i}}$ be the Galerkin finite element approximation solution of \eqref{N}, that is
\begin{equation}
\int_{\Omega_{i}}\nabla\bar{u}_{h_{i}}\cdot\nabla\bar{u}_{h_{i}}dx
+\int_{\Omega_{i}}\bar{u}_{h_{i}}\bar{u}_{h_{i}}dx=\int_{\Gamma}(-1)^{i+1}\mu_{h}\bar{u}_{h_{i}}d\tau.
\end{equation}
Summing over $i=1,2$ we find:
\begin{equation}\label{s45}
{||\bar{u}_{h}||}^{2}_{X}=-\int_{\Gamma}\mu_{h}[\![\bar{u}_{h}]\!]d\tau=-b(\bar{u}_{h},\mu_{h})
\end{equation}
and
\begin{equation}\label{s46}
{||\bar{u}-\bar{u}_{h}||}_{X}\leq C~h^{1/2}\sum^{2}_{i=1}{||\bar{u}_{i}||}_{H^{3/2}(\Omega_{i})}.
\end{equation}
From \eqref{s45}, we arrive at
\begin{equation}\label{s38}
\sup_{0\neq v_{h}\in X_{h}} -\frac{b(v_{h},\mu_{h})}{{||v_{h}||}_{X}}\geq
 -\frac{b(\bar{u}_{h},\mu_{h})}{{||\bar{u}_{h}||}_{X}}={||\bar{u}_{h}||}_{X}.
\end{equation}
Also, from triangle inequality, \eqref{s43}, \eqref{s44} and \eqref{s46}, we find
\begin{align}
{||\bar{u}_{h}||}_{X}\geq {||\bar{u}||}_{X}-{||\bar{u}-\bar{u}_{h}||}_{X}
\geq C_{1}{||\mu_{h}||}_{H^{-1/2}(\Gamma)}-C~h^{1/2}{||\mu_{h}||}_{L^{2}(\Gamma)}.\label{s39}
\end{align}
Hence, the Lemma follows from  \eqref{s38} and \eqref{s39}.\qquad
\end{proof}
\begin{lemma}
\label{bound}
 If $(v_{h},\mu_{h})\in X_{h}\times W_{h}$  then there exists $C>0$ such that
 \begin{align}
 &\mathcal{A}(u,\lambda;v_{h},\mu_{h})\leq C\big[{|||(u,\lambda)|||}^{2}
 +h{||\lambda||}^{2}_{L^{2}(\Gamma)}\big]^{1/2}{|||(v_{h},\mu_{h})|||}\label{s35}\\
 &\mathcal{A}(u_{h},\lambda_{h};v_{h},\mu_{h})\leq C{|||(u_{h},\lambda_{h})|||}~{|||(v_{h},\mu_{h})|||}.\label{s36}
 \end{align}
\end{lemma}
\begin{proof}
Applying Cauchy-Schwarz inequality and duality between $H^{-1/2}$ and $H^{1/2}$, we arrive at
\begin{align}
 \mathcal{A}(u,\lambda;v_{h},\mu_{h})&\leq \bigg[{||u||}^{2}_{X}+{||\lambda||}^{2}_{H^{-1/2}(\Gamma)}
 +h{||\lambda||}^{2}_{L^{2}(\Gamma)}\bigg]^{1/2}\nonumber\\
 &\times \bigg[{||v_{h}||}^{2}_{X}+{||\mu_{h}||}^{2}_{H^{-1/2}(\Gamma)}
+h{||\beta\nabla v_{h}\cdot\nu||}^{2}_{L^{2}(\Gamma)} +h{||\mu_{h}||}^{2}_{L^{2}(\Gamma)}\bigg]^{1/2}.\nonumber
 \end{align}
 Hence, \eqref{s35} follows by using Lemma \ref{lem2} and the hypothesis \eqref{streg}.
 In a similar way, we can derive \eqref{s36}.\qquad
\end{proof}

\begin{lemma}
\label{infsup1}
 There exists a  constant $C>0$ such that for $(w_{h},\phi_{h})\in X_{h}\times W_{h}$:
\begin{equation}\label{s37}
\sup_{(0,0)\neq (v_{h},\mu_{h})\in X_{h}\times W_{h}}\frac{\mathcal{A}(w_{h},\phi_{h};v_{h},\mu_{h})}{|||(v_{h},\mu_{h})|||}
\geq C{|||(w_{h},\phi_{h})|||}.
\end{equation}
\end{lemma}
\begin{proof}
Using \eqref{s9},  Lemma \ref{bound} and Young's inequality,  we find
\begin{align}
\mathcal{A}(w_{h},\phi_{h};-q_{h},0)
&=-\mathcal{A}(w_{h},0;q_{h},0)-\mathcal{A}(0,\phi_{h};q_{h},0)\nonumber\\
&\geq-C{||w_{h}||}_{X}{||q_{h}||}_{X}-b(\phi_{h},q_{h})-S\int_{\Gamma}\gamma\phi_{h}\{\!\!\{\beta\nabla q_{h}\cdot\nu\}\!\!\} d\tau\nonumber\\
&\geq -\frac{C}{2}\big(\frac{1}{c_{1}}{||w_{h}||}^{2}_{X}+c_{1}{||q_{h}||}^{2}_{X}\big)-b(\phi_{h},q_{h})\nonumber\\
&-\frac{\gamma}{2}\bigg(\frac{1}{c_{2}}{||\phi_{h}||}^{2}_{L^{2}(\Gamma)}+c_{2}{||\nabla q_{h}\cdot \nu||}^{2}_{L^{2}(\Gamma)}\bigg).
\end{align}
Now let $q_{h}\in X_{h}$ be the function for which supremum occurs in condition \eqref{s10} and assume that ${||q_{h}||}_{X}={||\phi_{h}||}_{H^{-1/2}(\Gamma)}$. Then using Lemma \ref{lem2},
\begin{align}
\mathcal{A}(w_{h},\phi_{h};-q_{h},0)&\geq -\frac{C}{2c_{1}}{||w_{h}||}^{2}_{X}-\big(c_{1}+\gamma_{0}C^{2}_{I}c_{2}/2\big)
{||\phi_{h}||}^{2}_{H^{-1/2}(\Gamma)}\nonumber\\
&-\frac{1}{2c_{2}}\gamma_{0}h{||\phi_{h}||}^{2}_{L^{2}(\Gamma)}
+\big(c_{4}{||\phi_{h}||}_{H^{-1/2}(\Gamma)}
-c_{5}\gamma{||\phi_{h}||}_{L^{2}(\Gamma)}\big){||\phi_{h}||}_{H^{-1/2}(\Gamma)}\nonumber\\
&\geq -c_{3}{||w_{h}||}^{2}_{X}+c_{4}{||\phi_{h}||}^{2}_{H^{-1/2}(\Gamma)}-c_{5}\gamma{||\phi_{h}||}^{2}_{L^{2}(\Gamma)}.
\end{align}
Let  $0<\alpha_{1}<\min(\alpha/c_{3},\alpha/c_{5})$.
Considering $(v_{h},\mu_{h})=(w_{h}-\alpha_{1} q_{h},-\phi_{h})$ and using the Lemma \ref{coer}, we find
\begin{align}
\mathcal{A}(w_{h},\phi_{h};v_{h},\mu_{h})&=\mathcal{A}(w_{h},\phi_{h};w_{h}-\alpha_{1} q_{h},-\phi_{h})\nonumber\\
&=\mathcal{A}(w_{h},\phi_{h};w_{h},-\phi_{h})+\alpha_{1}\mathcal{A}(w_{h},\phi_{h};-q_{h},0)\nonumber\\
&\geq (\alpha-\alpha_{1} c_{3}){||w_{h}||}^{2}_{X}+\alpha_{1} c_{4}{||\phi_{h}||}^{2}_{H^{-1/2}(\Gamma)}+(\alpha-\alpha_{1} c_{5})\gamma {||\phi_{h}||}^{2}_{L^{2}(\Gamma)}\nonumber\\
&\geq C({||w_{h}||}^{2}_{X}+{||\phi_{h}||}^{2}_{H^{-1/2}(\Gamma)})=C{|||(w_{h},\phi_{h})|||}^{2}.\label{s41}
\end{align}
Here, $c_{1},c_{2},c_{3},c_{4}$ and $c_{5}$ are positive constants.
Further,
\begin{align}
{|||(v_{h},\mu_{h})|||}^{2}&\leq {||w_{h}||}^{2}_{X}+\alpha_{1}^{2}{||q_{h}||}^{2}_{X}+{||\phi_{h}||}^{2}_{H^{-1/2}(\Gamma)}\nonumber\\
&\leq C{|||(w_{h},\phi_{h})|||}^{2}.\label{s42}
\end{align}
Hence,  \eqref{s37} follows from \eqref{s41} and \eqref{s42}.\qquad
\end{proof}
\vspace{0.1cm}
\begin{proof}\textbf{of Theorem \ref{main}}:
From Lemma \ref{infsup1}, there exist a pair $(v_{h},\mu_{h})\in X_{h}\times W_{h}$ such that
\begin{equation}\label{s32}
|||(v_{h},\mu_{h})|||<C
\end{equation}
and that implies
\begin{align}
|||(I_{h}u-u_{h},\Pi\lambda-\lambda_{h})|||
\leq\mathcal{A}(I_{h}u-u_{h},\Pi\lambda-\lambda_{h};v_{h},\mu_{h}).\label{s33}
\end{align}
 From \eqref{s33}  and orthogonality \eqref{s27} of $\mathcal{A}(\cdot,\cdot;\cdot,\cdot)$, we find
\begin{equation*}
|||(I_{h}u-u_{h},\Pi\lambda-\lambda_{h})|||
\leq\mathcal{A}(u-I_{h}u,\lambda-\Pi\lambda;v_{h},\mu_{h}).
\end{equation*}
From \eqref{s35} of Lemma \ref{bound},
\begin{align}
{|||(I_{h}u-u_{h},\Pi\lambda-\lambda_{h})|||}
&\leq C\big[{|||(u-I_{h}u,\lambda-\Pi\lambda)|||}^{2}\nonumber\\
&+h{||\lambda-\Pi\lambda||}^{2}_{L^{2}(\Gamma)}\big]^{1/2}{|||(v_{h},\mu_{h})|||}.\label{s34}
\end{align}
Note that
\begin{equation*}
{|||(u-u_{h},\lambda-\lambda_{h})|||}^{2}\leq 2\big({|||(u-I_{h}u,\lambda-\Pi\lambda)|||}^{2}
+{|||(I_{h}u-u_{h},\Pi\lambda-\lambda_{h})|||}^{2}\big).
\end{equation*}
Hence, from \eqref{s32}, \eqref{s34}, Lemma  \ref{inter} and Lemma \ref{proj},  the  \eqref{s40} follows.\qquad
\end{proof}

For the $L^2$-error estimate, we appeal to the Aubin-Nitsche duality
argument. Let $z_i={z_|}_{\Omega_i}\in H^2(\Omega_i)\cap
H^1_0(\Omega), i=1,2$ be the solution of the interface problem
\begin{eqnarray}\label{nisdual11}
-\nabla\cdot (\beta_i(x)\nabla z_i)+a_iz_i&=&u_{i}-u_{h_i}~~\mbox{in }\Omega_i,\\
z_i&=&0~~\mbox{ on }\partial\Omega\cap\partial\Omega_i,\\
{[\![z]\!]}=0,~~{[\![\beta \nabla z\cdot\nu]\!]}&=&0
~~~~\mbox{along}~~\Gamma,
\end{eqnarray}
which satisfies the regularity condition (see \cite{babu},
\cite{chen1})
\begin{eqnarray}
\displaystyle\sum_{i=1}^2{\|z_i\|}_{H^2(\Omega_i)}\leq
c{\|u-u_{h}\|}_{L^2(\Omega)}.\label{nisreg}
\end{eqnarray}
\begin{theorem}
\label{nisl2error}
Let $\mathcal{A}_{1}(\cdot,\cdot;\cdot,\cdot)$ denotes the form $\mathcal{A}(\cdot,\cdot;\cdot,\cdot)$ with $S=1$.
Also let $(u,\lambda)$ and $(u_{h},\lambda_{h})$ be the solutions of respective equations as in
Theorem \ref{main} with $S=1$. Then for $\gamma=\gamma_{0}h,~\gamma_{0}>0$, there exists a
positive constant $C$ independent of  $h$ and $u$ such that
\vspace{-0.1cm}
\begin{equation}\label{s19}
{||u-u_{h}||}_{L^{2}(\Omega)}\leq C~h^{2}\sum^{2}_{i=1}{||u||}_{H^{2}(\Omega_{i})}.
\end{equation}
\end{theorem}
\vspace{-0.1cm}
\begin{proof}
Clearly $z$ satisfies $\mathcal{A}_{1}(z,\lambda_{z};u-u_{h},\lambda-\lambda_{h})
=(u-u_{h},u-u_{h})$, where $\lambda_{z}=\beta_{1}\nabla z_{1}\cdot\nu=
-\beta_{2}\nabla z_{2}\cdot\nu$ is the Lagrange multiplier. By symmetric property and  orthogonality,
\begin{align}
{||u-u_{h}||}^{2}_{L^{2}(\Omega)}&=\mathcal{A}_{1}(z-I_{h}z,\lambda_{z}
-\Pi\lambda_{z};u-u_{h},\lambda-\lambda_{h})\nonumber\\
&=\sum^{2}_{i=1}\int_{\Omega_{i}}\big(\beta_{i}\nabla(z_{i}-I_{h_{i}}z_{i})\cdot\nabla (u-u_{h})
+a_{i}(x)(z_{i}-I_{h_{i}}z_{i})(u-u_{h})\big)dx\nonumber\\
&+\int_{\Gamma}[\![u-u_{h}]\!](\lambda_{z}-\Pi\lambda_{z})d\tau+\int_{\Gamma}\gamma(\lambda_{z}-\Pi\lambda_{z})
\big\{\!\!\big\{\beta\nabla(u-u_{h})\cdot\nu\big\}\!\!\big\}d\tau\nonumber\\
&+\int_{\Gamma}[\![z-I_{h}z]\!](\lambda-\lambda_{h})d\tau-\int_{\Gamma}\gamma\big
\{\!\!\big\{\beta\nabla(z-I_{h}z)\cdot\nu\big\}\!\!\big\}\big\{\!\!\big\{\beta\nabla (u-u_{h})\cdot\nu\big\}\!\!\big\}d\tau\nonumber\\
&+\int_{\Gamma}\gamma\big\{\!\!\big\{\beta\nabla(z-I_{h}z)\cdot\nu\big\}\!\!\big\}
(\lambda-\lambda_{h})d\tau-\int_{\Gamma}\gamma(\lambda_{z}
-\Pi\lambda_{z})(\lambda-\lambda_{h})d\tau.\nonumber
\end{align}
\vspace{-0.1cm}
Using Cauchy-Schwarz inequality, the duality paring between $H^{-1/2}$ and $H^{1/2}$ and trace inequality, we find
\vspace{-0.1cm}
\begin{align}
 {||u-u_{h}||}^{2}_{L^{2}(\Omega)}
 \leq C&\bigg[{||z-I_{h}z||}_{X}{||u-u_{h}||}_{X}
 +{||u-u_{h}||}_{X}{||\lambda_{z}-\Pi\lambda_{z}||}_{H^{-1/2}(\Gamma)}\nonumber\\
 &+\gamma{||\lambda_{z}-\Pi\lambda_{z}||}_{L^{2}(\Gamma)}
 {\big|\big|\beta\nabla(u-u_{h})\cdot\nu\big|\big|}_{L^{2}(\Gamma)}\nonumber\\
 &+\gamma^{-1/2}{||[\![z-I_{h}z]\!]||}_{L^{2}(\Gamma)}\gamma^{1/2}
 {||\lambda-\lambda_{h}||}_{L^{2}(\Gamma)}\nonumber\\
 &+\gamma{\big|\big|\beta\nabla(z-I_{h}z)\cdot\nu\big|\big|}_{L^{2}(\Gamma)}
 {\big|\big|\beta\nabla(u-u_{h})\cdot\nu\big|\big|}_{L^{2}(\Gamma)}\nonumber\\
 &+\gamma{\big|\big|\beta\nabla(z-I_{h}z)\cdot\nu\big|\big|}_{L^{2}(\Gamma)}
 {||\lambda-\lambda_{h}||}_{L^{2}(\Gamma)}\nonumber\\
 &+\gamma{||\lambda_{z}-\Pi\lambda_{z}||}_{L^{2}(\Gamma)}{||\lambda-\lambda_{h}||}_{L^{2}(\Gamma)}\bigg]\nonumber\\
 &\leq C\bigg[
 {||z-I_{h}z||}^{2}_{X}+{||\lambda_{z}-\Pi\lambda_{z}||}^{2}_{H^{-1/2}(\Gamma)}
 +\gamma{||\lambda_{z}-\Pi\lambda_{z}||}^{2}_{L^{2}(\Gamma)}\nonumber\\
&+\gamma^{-1}{||[\![z-I_{h}z]\!]||}^{2}_{L^{2}(\Gamma)}
 +\gamma{\big|\big|\beta\nabla(z-I_{h}z)\cdot\nu\big|\big|}^{2}_{L^{2}(\Gamma)}\bigg]^{1/2}\nonumber\\
 &\times
 \bigg[
{||u-u_{h}||}^{2}_{X}+\gamma{||\lambda-\lambda_{h}||}^{2}_{L^{2}(\Gamma)}
+\gamma{\big|\big|\beta\nabla (u-u_{h})\cdot\nu\big|\big|}^{2}_{L^{2}(\Gamma)}
 \bigg]^{1/2}.\label{l1}
 \end{align}
 Using trace inequality and Lemma \ref{lem2}, we find
\begin{align}
\gamma{\big|\big|\beta\nabla(u-u_{h})\cdot\nu\big|\big|}^{2}_{L^{2}(\Gamma)}
&\leq C\bigg[\gamma{\big|\big|\beta\nabla(u-I_{h}u)\cdot\nu\big|\big|}^{2}_{L^{2}(\Gamma)}
+{\big|\big|\gamma^{1/2}\beta\nabla(I_{h}u-u_{h})
\cdot\nu\big|\big|}^{2}_{L^{2}(\Gamma)}\bigg]\nonumber\\
&\leq C\bigg[
\gamma\sum^{2}_{i=1}{||u_{i}-I_{h_{i}}u_{i}||}^{2}_{H^{3/2}(\Omega_{i})}
+\sum^{2}_{i=1}{||\nabla(I_{h_{i}}u_{i}-u_{h_{i}})||}^{2}_{L^{2}(\Omega_{i})}
\bigg]\nonumber\\
&\leq C\bigg[
\gamma\sum^{2}_{i=1}{||u_{i}-I_{h_{i}}u_{i}||}^{2}_{H^{3/2}(\Omega_{i})}
+{||u-I_{h}u||}^{2}_{X}+{||u-u_{h}||}^{2}_{X}
\bigg].\label{l2}
\end{align}
From \eqref{l1} and \eqref{l2} using Theorem \ref{main}, Lemma \ref{inter} and Lemma \ref{proj} with $\gamma=\gamma_{0}h$, we arrive at
\begin{equation}
{||u-u_{h}||}^{2}_{L^{2}(\Omega)}\leq C~h^{2}\sum^{2}_{i=1}{||z_{i}||}_{H^{2}(\Omega_{i})}\sum^{2}_{i=1}{||u_{i}||}_{H^{2}(\Omega_{i})}.
\end{equation}
Hence,  \eqref{s19} follows by using the regularity condition \eqref{nisreg}.\qquad
\end{proof}
\begin{remark}
 For unsymmetric case, when $S\in [0,1)$, the  $L^{2}$-error estimate is of $O(h^{3/2})$. However,
 with an additional assumption on the interpolants $I_{h}$ and $\Pi$ \textrm{i.e.},
\begin{equation}
h\bigg({\big|\big|\nabla I_{h}z\cdot\nu\big|\big|}^{2}_{L^{2}(\Gamma)}+{||\Pi\lambda_{z}||}^{2}_{L^{2}(\Gamma)}\bigg)\leq C~h^{2}
{\big|\big|\nabla z\cdot\nu\big|\big|}^{2}_{H^{1/2}(\Gamma)}
\end{equation}
we can establish optimal order of $L^{2}$-estimate in a similar way as in Barbosa \textrm{et al.} (see \cite{barbosa1}).
\end{remark}
\section{Matrix formulation}
The stabilized Nitsche's mortaring method \eqref{s16} can be represented in matrix form  by $AU=F$.

 The stiffness matrix
\begin{equation*}
A=\left(
    \begin{array}{ccccc}
      A^{1}_{ii} & A^{1}_{is} & 0 & 0 & 0 \\
      A^{1}_{si} & A^{1}_{ss} & 0 & 0 & Q_{s}+\frac{S\gamma}{2}R^{s}-\frac{S\gamma}{4}R^{s}_{n} \\
      0 & 0 & A^{2}_{ii} & A^{2}_{im} & 0 \\
      0 & 0 & A^{2}_{mi} & A^{2}_{mm} & -(Q_{s})^{T}+\frac{S\gamma}{2}(R^{m})^{T}-\frac{S\gamma}{4}(R^{m}_{n})^{T} \\
      0 & Q_{s}+\frac{\gamma}{2}R^{s} & 0 & -Q_{m}+\frac{\gamma}{2}R^{m} & -\gamma Q_{mm} \\
    \end{array}
  \right),
\end{equation*}

 where, for $1\leq l\leq 2$,
 \begin{align*}
 &(A^{l}_{ii})_{ij}=a(\phi^{(l)}_{i},\phi^{(l)}_{j}), \hspace{0.2cm}x^{l}_{i},x^{l}_{j}\in\Omega_{l},\hspace{0.2cm}
 (A^{1}_{is})_{ij}=a(\phi^{(1)}_{i},\phi^{s}_{j}), \hspace{0.2cm}x^{1}_{i}\in\Omega_{1},x^{s}_{j}\in\Gamma_{1},\\ &(A^{1}_{ss})_{ij}=a(\phi^{s}_{i},\phi^{s}_{j}),\hspace{0.2cm}x^{s}_{i},x^{s}_{j}\in\Gamma_{1},\hspace{0.2cm}
 (A^{2}_{im})_{ij}=a(\phi^{(2)}_{i},\phi^{m}_{j}),\hspace{0.2cm} x^{2}_{i}\in\Omega_{2}, x^{m}_{j}\in\Gamma_{2},\\ &(A^{2}_{mm})_{ij}=a(\phi^{m}_{i},\phi^{m}_{j}),\hspace{0.2cm}x^{m}_{i},x^{m}_{j}\in\Gamma_{2},\hspace{0.2cm}
 (Q_{s})_{ij}=\int_{\Gamma_{1}}\phi^{s}_{j}\psi_{i}d\tau,\hspace{0.2cm}(Q_{m})_{ij}
 =\int_{\Gamma_{2}}\psi^{m}_{j}\psi_{i}d\tau,\\
 &(Q_{mm})_{ij}=\int_{\Gamma_{2}}\psi_{i}\psi_{j}d\tau,
  (R^{s})_{ij}=\int_{\Gamma_{1}}\beta_{1}\nabla\phi^{s}_{j}\cdot \nu_{1}\psi_{i}d\tau,
 (R^{m})_{ij}=\int_{\Gamma_{2}}\beta_{2}\nabla\phi^{m}_{j}\cdot \nu_{2}\psi_{i}d\tau,\\
  &(R^{s}_{n})_{ij}=\int_{\Gamma_{1}}\beta_{1}\nabla\phi^{s}_{j}
 \cdot\nu_{1}\beta_{1}\nabla\phi^{s}_{i}\cdot\nu_{1}d\tau,
  (R^{m}_{n})_{ij}=\int_{\Gamma_{2}}\beta_{2}\nabla\phi^{m}_{j}\cdot\nu_{2}
 \beta_{2}\nabla\phi^{m}_{i}\cdot\nu_{2}d\tau.
\end{align*}
  Also, $U=\big(u^{1}_{i},u^{1}_{s},u^{2}_{i},u^{2}_{m},\lambda_{m}\big)^{T}$. The unknowns $u^{1}_{i}$ and $u^{2}_{i}$ are
  associated with the internal nodes in $\Omega_{1}$ and $\Omega_{2}$ respectively.   Unknown $u^{1}_{s}$ and $u^{1}_{m}$ are
  associated with $\Gamma_{1}$ and $\Gamma_{2}$ and $\lambda_{m}$ are the unknown Lagrange multipliers associated with $W_{h}$.

The load vector $F=(F^{1}_{i},F^{1}_{s},F^{2}_{i}, F^{2}_{m}, 0)^{T},$ where,
\begin{align*}
  &(F^{l}_{i})_{i}=(f^{l}_{i},\phi^{l}_{i}),\hspace{0.2cm}x^{l}_{i}\in\Omega_{l},\hspace{0.2cm} (F^{l}_{s})_{i}=(f^{l}_{i},\phi^{s}_{i}),\hspace{0.2cm}x^{l}_{i}\in\Omega_{l},x^{s}_{i}\in\Gamma_{1},\\ &(F^{l}_{s})_{i}=(f^{l}_{i},\phi^{m}_{i}),\hspace{0.2cm}x^{l}_{i}\in\Omega_{l},x^{m}_{i}\in\Gamma_{2}.
\end{align*}
 Here, $\psi_{i}$ are the nodal basis functions for $W_{h}$, $\phi^{s}_{j}$ and $\phi^{m}_{j}$ are the basis functions for $W_{h_{1}}$ and $W_{h_{2}}$ respectively.
\section{Numerical experiments}
We choose problem \eqref{s1}-\eqref{s4} over the unit square domain $\Omega=(0,1)\times (0,1)$.
We divide the domain $\Omega$  into two
equal subdomains $\Omega_{i},~ i=1,2$ (see Figure \ref{soln}). Each subdomains further subdivided into linear triangular elements of different mesh size $h_{i}$. We choose the penalty parameter to be $\gamma=O(h)$.  Set  $a_{1}$ and $a_{2}$ to be zero. We choose  $f$ such that the exact
solution of the problem is $u(x,y)=sin^{2} \pi x ~sin^{2}\pi y$.

The order of convergence `$p$' for the error ${||u-u_{h}||}_{L^{2}(\Omega)}$ and  the order of convergence `$q$' for the error ${||\lambda-\lambda_{h}||}_{L^{2}(\Gamma)}$ with respect to  the discretization parameter $h$ are computed by taking discontinuous coefficients pairs $(\beta_{1},\beta_{2})=(1,10)$, $(1,10^7)$, $(10^7,10^{-7})$ in the subdomains $\Omega_{1}$ and $\Omega_{2}$, see Tables \ref{table1}, \ref{table3}, \ref{table2}. Figure \ref{Order of Convergence} (a) shows the computed order of convergence for ${||u-u_{h}||}_{L^{2}(\Omega)}$ with respect to $h$ in the log-log scale.   Figure \ref{Order of Convergence} (b) shows the convergence rate of the Lagrange multiplier with respect to $h$. Note that, since
the exact solution is smooth, the convergence rates of error ${||u-u_{h}||}_{L^{2}(\Omega)}$ and  ${||\lambda-\lambda_{h}||}_{L^{2}(\Gamma)}$ are computationally obtained  as expected \textrm{i.e.}, $O(h^{2})$ and $O(h)$ respectively.
\begin{figure}[h]
\centering
\includegraphics*[width=13cm,height=7cm]{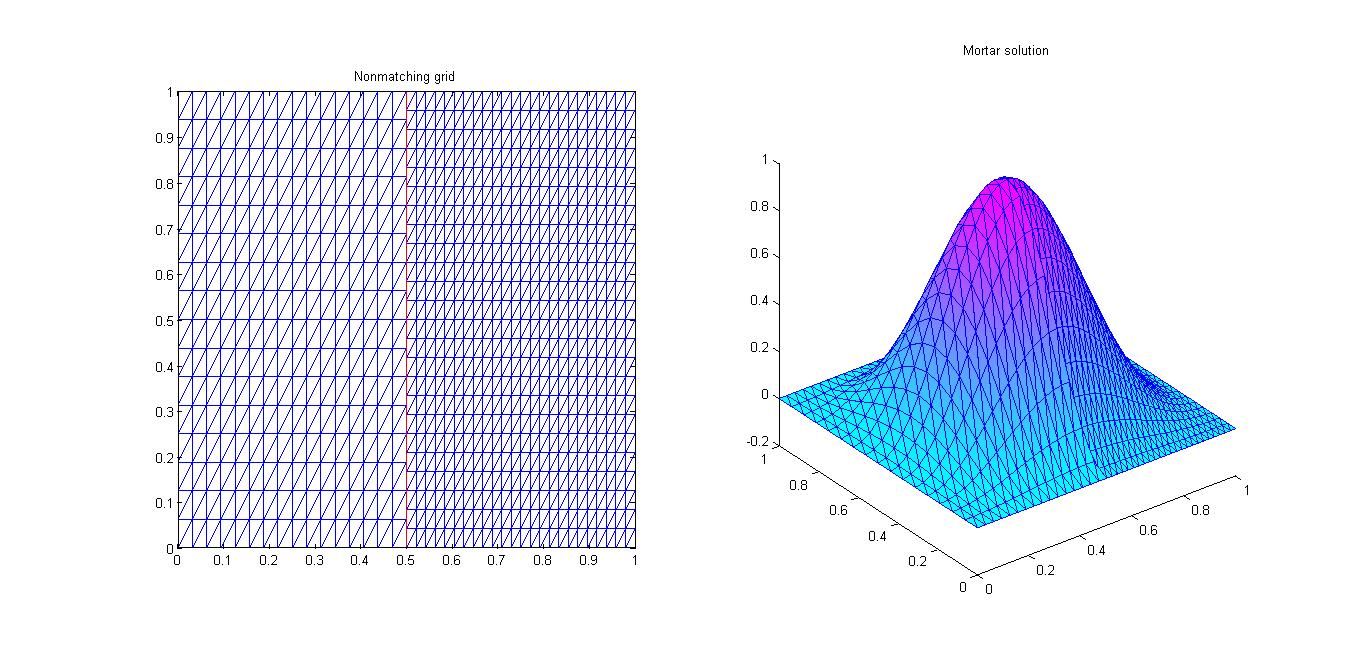}
\caption{Nonmatching grid and computed mortar solution with $(\beta_1,\beta_2)=(1,10)$ at refinement level $(h_1,h_2)=(1/16,1/24)$.}
\label{soln}
\end{figure}
\begin{table}[hbtp]
\caption{Order of
Convergence of ${||u-u_{h}||}_{L^{2}(\Omega)}$ and ${||\lambda-\lambda_{h}||}_{L^{2}(\Gamma)}$ with 
$\beta_1=1,\beta_2=10$}
\centering
\begin{tabular}{lclclc}
\hline
$(h_1,h_2)$&$h$&${\|e_u\|}_{L^2(\Omega)}$&${\|e_\lambda\|}_{L^2(\Gamma)}$&$p$&$q$\\
\hline
$(\frac{1}{4},\frac{1}{6})$&1/4&0.062158&0.43687&\\
$(\frac{1}{8},\frac{1}{12})$&1/8&0.016638&0.24098&1.901458062028300& 0.858290623104806\\
$(\frac{1}{16},\frac{1}{24})$&1/16&0.0041882&0.11229&1.990079779863557&1.101683961685886\\
$(\frac{1}{32},\frac{1}{48})$&1/32&0.0010422&0.049426& 2.006698177352890& 1.183887393718836\\
$(\frac{1}{64},\frac{1}{96})$&1/64&0.00025934&0.021541& 2.006715513479729&  1.198184929427100\\
\hline
\end{tabular}
\label{table1}
\end{table}

\begin{table}[hbtp]
\caption{Order of
Convergence of ${||u-u_{h}||}_{L^{2}(\Omega)}$  and ${||\lambda-\lambda_{h}||}_{L^{2}(\Gamma)}$ with 
$\beta_1=1,\beta_2=10^7$}
\centering
\begin{tabular}{lclclc}
\hline
$(h_1,h_2)$&$h$&${\|e_u\|}_{L^2(\Omega)}$&${\|e_\lambda\|}_{L^2(\Gamma)}$&$p$&$q$\\
\hline
$(\frac{1}{4},\frac{1}{6})$&1/4&0.061619&0.45209&\\
$(\frac{1}{8},\frac{1}{12})$&1/8&0.016431&0.25338& 1.906954983214551&  0.835307353359031  \\
$(\frac{1}{16},\frac{1}{24})$&1/16&0.0041218&0.11953&1.995073877668074&  1.083929897366207\\
$(\frac{1}{32},\frac{1}{48})$&1/32&0.0010233&0.052978 & 2.010045342634421&   1.173907469688573\\
$(\frac{1}{64},\frac{1}{96})$&1/64&0.00025429&0.023106 & 2.008682526770985&  1.197125851836898 \\
\hline
\end{tabular}
\label{table3}
\end{table}
\begin{table}[hbtp]
\caption{Order of
Convergence of ${||u-u_{h}||}_{L^{2}(\Omega)}$  and ${||\lambda-\lambda_{h}||}_{L^{2}(\Gamma)}$ with 
$\beta_1=10^{-7},\beta_2=10^7$}
\centering
\begin{tabular}{lclclc}
\hline
$(h_1,h_2)$&$h$&${\|e_u\|}_{L^2(\Omega)}$&${\|e_\lambda\|}_{L^2(\Gamma)}$&$p$&$q$\\
\hline
$(\frac{1}{4},\frac{1}{6})$&1/4&0.058294&1.8372e-07&\\
$(\frac{1}{8},\frac{1}{12})$&1/8&0.015655&6.2439e-08& 1.896723890974834&  1.556989351729802\\
$(\frac{1}{16},\frac{1}{24})$&1/16&0.0040013&2.366e-08& 1.968082803651483&1.399997358151351\\
$(\frac{1}{32},\frac{1}{48})$&1/32&0.0010066&1.032e-08&  1.990978296765009& 1.197007102916534\\
$(\frac{1}{64},\frac{1}{96})$&1/64&0.00025206&5.325e-09& 1.997651406176591&  0.954589540310056\\
\hline
\end{tabular}
\label{table2}
\end{table}
\begin{figure}
\centering \mbox{
%\begin{subfigmatrix}{2}% number of columns
        \subfigure[Order of Convergence of $||u-u_{h}||_{L^{2}(\Omega)}$ w.r.t.  $h$.]
                {
                        \includegraphics[height=5.0cm, width=6.5cm]{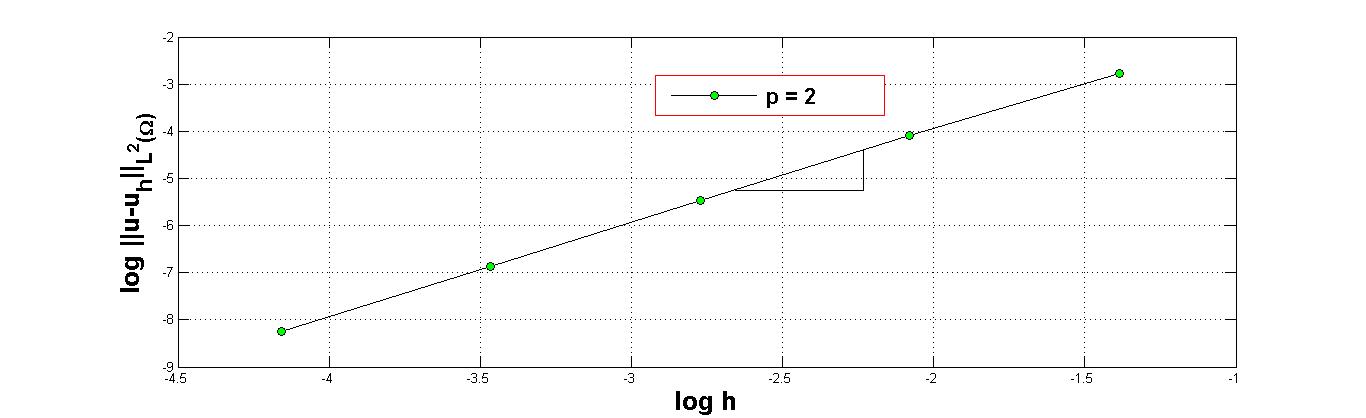}
                }
        \subfigure[Order of Convergence of $||\lambda-\lambda_{h}||_{L^{2}(\Gamma)}$ w.r.t. $h$.]
                {
                        \includegraphics[height=5.0cm, width=6.5cm]{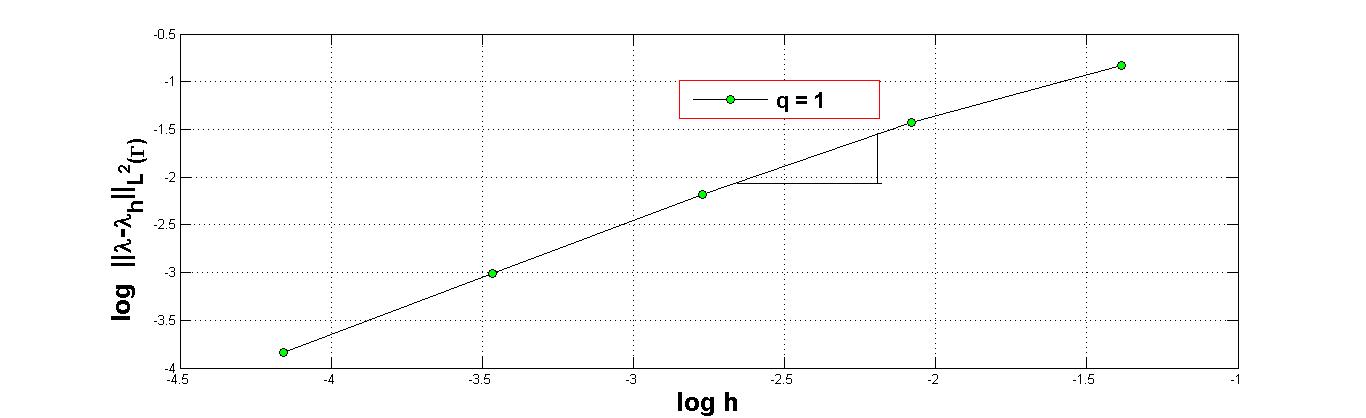}
                }}
%\end{subfigmatrix}
\caption{Order of Convergence with discontinuous coefficients.  $\beta_1=1,\beta_2=10$}
\label{Order of Convergence}
\end{figure}
\newpage
\section{Conclusion}
 In order to alleviate the inf-sup condition in the mortar method with Lagrange multiplier,
 a stabilized method is presented and optimal error estimates are  obtained in natural norm
 which is independent of mesh. Numerical experiments presented here depict the performance
 of the method and supports the theoretical error estimates.  Here, the multipliers are simply
 the nodal basis functions restricted to the interface, one can consider  global polynomials as
 multipliers as in \cite{Hansbo} to avoid the cumbersome integration over unrelated meshes.
% \section*{Acknowledgment}
% We thank the referees for their valuable comments and suggestions which helped us in improving the quality of the paper.

%    Text of article.

%    Bibliographies can be prepared with BibTeX using amsplain,
%    amsalpha, or (for "historical" overviews) natbib style.
%\bibliographystyle{amsplain}
%    Insert the bibliography data here.

\end{document}